\documentclass[12pt]{amsart}
\pdfoutput=1
\usepackage{latexsym, amsmath, amssymb, amsthm, mathrsfs, bm, xcolor, parskip}
\usepackage[mathscr]{eucal}
\usepackage{paralist, array}
\usepackage{tikz}
\usetikzlibrary{arrows,shapes,positioning,calc,patterns,cd}
\usepackage{graphicx}
\usepackage[alphabetic]{amsrefs}
\usepackage[T1]{fontenc}

\usepackage[colorlinks=true,linkcolor=blue,urlcolor=blue, citecolor=blue]%
  {hyperref}
\usepackage[centering, includeheadfoot, hmargin=1.2in, tmargin=1.5in, 
  bmargin=1.5in, headheight=7pt]{geometry}
\newtheorem{lemma}{Lemma}[section]
\newtheorem{theorem}[lemma]{Theorem}

\newtheorem{proposition}[lemma]{Proposition}
\theoremstyle{definition}
\newtheorem{definition}[lemma]{Definition}
\newtheorem{remark}[lemma]{Remark}
\newtheorem{examplex}[lemma]{Example}
\newenvironment{example}
  {\pushQED{\qed}\examplex}
  {\popQED\endexamplex}

\newcommand{\kk}{\ensuremath{\Bbbk}} 

\newcommand{\PP}{\ensuremath{\mathbb{P}}}

\newcommand{\ZZ}{\ensuremath{\mathbb{Z}}} 

\newcommand{\cO}{\ensuremath{\mathcal{O}}}

\DeclareMathOperator{\codim}{codim}
\DeclareMathOperator{\coker}{coker}
\DeclareMathOperator{\Cox}{Cox}

\DeclareMathOperator{\Pic}{Pic}

\DeclareMathOperator{\rank}{rank}

\DeclareMathOperator{\Spec}{Spec}

\DeclareMathOperator{\depth}{depth}

\DeclareMathOperator{\Fitt}{Fitt}
\DeclareMathOperator{\Div}{Div}

\makeatletter
\def\bbordermatrix#1{\begingroup \m@th
  \@tempdima 4.75\p@
  \setbox\z@\vbox{%
    \def\cr{\crcr\noalign{\kern2\p@\global\let\cr\endline}}%
    \ialign{$##$\hfil\kern2\p@\kern\@tempdima&\thinspace\hfil$##$\hfil
      &&\quad\hfil$##$\hfil\crcr
      \omit\strut\hfil\crcr\noalign{\kern-\baselineskip}%
      #1\crcr\omit\strut\cr}}%
  \setbox\tw@\vbox{\unvcopy\z@\global\setbox\@ne\lastbox}%
  \setbox\tw@\hbox{\unhbox\@ne\unskip\global\setbox\@ne\lastbox}%
  \setbox\tw@\hbox{$\kern\wd\@ne\kern-\@tempdima\left[\kern-\wd\@ne
    \global\setbox\@ne\vbox{\box\@ne\kern2\p@}%
    \vcenter{\kern-\ht\@ne\unvbox\z@\kern-\baselineskip}\,\right]$}%
  \null\;\vbox{\kern\ht\@ne\box\tw@}\endgroup}
\makeatother

\begin{document}

\vspace*{3.0em}

\title[What makes a complex a virtual resolution?]{What makes a complex a virtual resolution?}
\author[M.C. Loper]{Michael C. Loper}
\address{Department of Mathematics, University of
  Wisconsin - River Falls}
  \email{michael.loper@uwrf.edu}

\begin{abstract}
Virtual resolutions are homological representations of finitely generated $\Pic(X)$-graded modules over the Cox ring of a smooth projective toric variety. In this paper, we identify two algebraic conditions that characterize when a chain complex of graded free modules over the Cox ring is a virtual resolution. We then turn our attention to the saturation of Fitting ideals by the irrelevant ideal of the Cox ring and prove some results that mirror the classical theory of Fitting ideals for Noetherian rings.
\end{abstract}

\maketitle

\setcounter{section}{0}

\section{Introduction}
\label{sec:intro}
In a famous paper, Buchsbaum and Eisenbud present two criteria that completely determine whether or not a chain complex is exact over a Noetherian ring \cite{be}. This is done without examining the homology of the complex. These criteria are useful in investigating a module by examining the minimal free resolution.

In turn, these criteria can be used to study the geometry of projective space. Coherent sheaves over projective space correspond to finitely generated graded modules over a standard-graded polynomial ring. Properties of this module and the sheaf associated to the graded module such as degree, dimension, and Hilbert polynomial, are encoded in the minimal free resolution of the graded module.

Before stating the main theorem from \cite{be}, we must introduce some notation. A map of free $S$-modules $\varphi\colon  F \rightarrow G$ can be expressed as a matrix with entries in $S$ by choosing bases of $F$ and $G$. Denote by $I_r(\varphi)$ the ideal generated by the $r \times r$ minors of $\varphi$. Then $\rank(\varphi)$ will be the largest $r$ such that $I_r(\varphi) \ne 0$. The ideal $I_{\rank(\varphi)}(\varphi)$ will be the most important of these ideals of minors, and we set $I(\varphi) := I_{\rank(\varphi)}(\varphi)$. By convention, we define $I_k(\varphi) = S$ for every integer $k \le 0$.

In fact, these ideals of minors can be extended to projective modules (that may not be finitely generated). Indeed $\varphi$ gives rise to a map $\bigwedge^r \varphi\colon  \bigwedge^r F \rightarrow \bigwedge^r G$, and the rank of $\varphi$ is the largest $r$ such that $\bigwedge^r \varphi \ne 0$. In this context, $I_r(\varphi)$ is the image of the map $\bigwedge^r F \otimes (\bigwedge^r G)^* \rightarrow S$. The main theorem from \cite{be} can now be stated.

\begin{theorem}[\cite{be}]
\label{thm:be}
Let $R$ be a Noetherian ring. Suppose
\[F_\bullet :=  \left[F_0 \xleftarrow{\varphi_1} F_{1} \xleftarrow{\varphi_2} F_2 \xleftarrow{\varphi_3} \cdots \xleftarrow{\varphi_n} F_n \longleftarrow 0 \right] \]
is a chain complex of free $R$-modules. Then $F_\bullet$ is exact if and only if both of the following conditions are satisfied:
  \begin{enumerate}[(a)]
  \item $\rank(\varphi_{i}) + \rank(\varphi_{i+1}) = \rank(F_i)$ (taking $\varphi_{n+1} = 0$),
 
  \item $\depth(I(\varphi_i)) \ge i$
  \end{enumerate}
  for each $i = 1,2,...,n$.
\end{theorem}

When the toric variety is projective space, the locally free resolutions of coherent sheaves over $\PP^N$ and the free resolutions of $\Cox(\PP^N)$-modules coincide. Unfortunately when studying coherent sheaves over more general smooth projective toric varieties, the situation is not as well-behaved. Locally free resolutions of a coherent sheaf are often shorter and thinner than the corresponding minimal free resolutions of the modules. Tying these concepts more closely together, Berkesch, Erman, and Smith introduced the notion of virtual resolutions in \cite{bes}. The main theorem in the present paper (Theorem~\ref{mainthm}) is the virtual analogue to the main theorem of \cite{be} (Theorem~\ref{thm:be}).

\subsection*{Notation}
Throughout this paper, $X$ will be a smooth projective toric variety and $S$ will denote the Cox ring of $X$ over an algebraically closed field $\kk$. The Cox ring $S$ is graded by the Picard group of $X$, which we denote by $\Pic(X)$ \cite{cox}*{\textsection 1}. In particular, $S$ is a polynomial ring with a multigrading by $\ZZ^r$ for $r = \rank \; \Pic(X)$. Let $B$ denote the irrelevant ideal of $S$, which is radical; $M$ will be a finitely generated $\Pic(X)$-graded module over $S$ and $\widetilde{M}$ denotes the sheaf of $M$ over $X$, as constructed in \cite{cox}*{\textsection 3}. Given an ideal $I$ of $S$, we denote the set of all homogenous prime ideals containing $I$ by $V(I)$.

\begin{definition}
  \label{def:vres}
  A graded free complex
  \[F_\bullet :=  [F_0 \xleftarrow{\varphi_1} F_1 \longleftarrow \cdots \longleftarrow F_{n-1} \xleftarrow{\varphi_n} F_n \longleftarrow \cdots]\]
of $S$-modules is called a \textbf{virtual resolution} of $M$ if the corresponding complex of vector bundles $\widetilde{F_\bullet}$ is a locally free resolution of the sheaf $\widetilde{M}$.
\end{definition}

The above definition uses the geometric language, but virtual resolutions can be equivalently defined algebraically. The $\Pic(X)$-graded $S$-module associated to a sheaf $\mathcal{F}$ over $X$ is defined to be
\[\Gamma_*(\mathcal{F}) = \bigoplus_{\alpha \in \Pic(X)} \Gamma(X,\mathcal{F}(\alpha))\]
(see \cite{cox}*{Theorem~3.2}). The complex $F_\bullet$ is a virtual resolution of $M$ if
\[\Gamma_*\big(\widetilde{M}\big) \cong \Gamma_*(\widetilde{\coker(\varphi_1)})\]
and for every $i \ge 1$, there is an $\ell$ such that $B^\ell H_i(F) = 0$, where again $B$ is the irrelevant ideal of $S$. 

The main result of this paper is the following theorem.
\begin{theorem}
\label{mainthm}
Let $X$ be a smooth projective toric variety with $S = \Cox(X)$. Suppose
\[F_\bullet :=  [F_0 \xleftarrow{\varphi_1} F_1 \longleftarrow \cdots \longleftarrow F_{n-1} \xleftarrow{\varphi_n} F_n \longleftarrow 0]\]
is a $\Pic(X)$-graded complex of free $S$-modules. Then $F_\bullet$ is a virtual resolution if and only if both of the following conditions are satisfied:
  \begin{enumerate}[(a)]
  \item $\rank(\varphi_{i}) + \rank(\varphi_{i+1}) = \rank(F_i)$ (with $\varphi_{n+1} = 0$),
 
  \item $\depth(I(\varphi_i):B^\infty) \ge i$
  \end{enumerate}
  for each $i = 1,2,...,n$.
\end{theorem}

As in \cite{be}, we assign the unit ideal infinite depth, so that condition $(b)$ holds if $I(\varphi_i):B^\infty = S$, i.e., $I(\varphi_i)$ is irrelevant. The difference between Theorem~\ref{thm:be} and Theorem~\ref{mainthm} is the replacement of exactness with virtuality and the addition of the saturation of ideals of minors by the irrelevant ideal $B$. Of course, any complex of graded free $S$-modules that is exact will also be a virtual resolution. Further, if a complex is exact, then the conditions of Theorem~\ref{mainthm} will be satisfied by Theorem~\ref{thm:be}. On the other hand, below is an example of a complex that is a virtual resolution but is not exact.

\begin{example}
  Let $X = \PP^1 \times \PP^2$ so that $S = \kk[x_0,x_1,y_0,y_1,y_2]$ with $\deg(x_i) = (1,0)$ and $\deg(y_i) = (0,1)$. Then the irrelevant ideal $B$ is  $\langle x_0,x_1\rangle \cap \langle y_0, y_1, y_2 \rangle$. Let $I$ be the following $B$-saturated ideal of 4 points:
  \[\begin{matrix} \langle
    {y}_{0} {y}_{1}-22 {y}_{1}^{2}-20 {y}_{0}
      {y}_{2}-26 {y}_{1} {y}_{2}-6 {y}_{2}^{2}, \\{x}_{0}
      {y}_{0}+43 {x}_{0} {y}_{2}-13 {x}_{1} {y}_{0}+50 {x}_{1}
      {y}_{1}+20 {x}_{1} {y}_{2},\\
      {x}_{0} {y}_{1}+10 {x}_{0} {y}_{2}-{x}_{1}
      {y}_{0}-19 {x}_{1} {y}_{1}+47 {x}_{1} {y}_{2}, \\{y}_{0}^{2}-20
      {y}_{1}^{2}-14 {y}_{0} {y}_{2}-22 {y}_{1} {y}_{2}+33
      {y}_{2}^{2},\\
      {x}_{0}^{2} {y}_{2}-42 {x}_{0}
      {x}_{1} {y}_{2}+42 {x}_{1}^{2} {y}_{0}-48 {x}_{1}^{2}
      {y}_{1}+13 {x}_{1}^{2} {y}_{2}, \\{x}_{0}^{4}+27 {x}_{0}^{3}
      {x}_{1}+25 {x}_{0}^{2} {x}_{1}^{2}-29 {x}_{0} {x}_{1}^{3}-40
      {x}_{1}^{4}
\rangle. \end{matrix}\]

The minimal free resolution of $S/I$ is included below, calculated using \textit{Macaulay2} \cite{M2}.
\[S \longleftarrow \begin{matrix} S(-1,-1)^2 \\ \oplus \\ S(0,-2)^2 \\ \oplus \\ S(-2,-1) \\ \oplus \\ S(-4,0) \end{matrix}
  \longleftarrow \begin{matrix} S(-1,-2)^2 \\ \oplus \\ S(-2,-2)^3 \\ \oplus \\ S(-1,-3)^2 \\ \oplus \\ S(0,-4) \\ \oplus \\ S(-1,-4)^3 \end{matrix}
  \longleftarrow \begin{matrix} S(-2,-3)^3 \\ \oplus \\ S(-1,-4)^2 \\ \oplus \\ S(-4,-2)^3 \end{matrix}
  \longleftarrow \begin{matrix} S(-2,-4) \\ \oplus \\ S(-4,-3) \end{matrix} \longleftarrow 0. \]
  
\cite{bes}*{Theorem~1.3} states that if $X = \PP^{n_1} \times \PP^{n_2} \times \cdots \times \PP^{n_r}$, the $S = \Cox(X)$-ideal $I$ is $B$-saturated (as it is in this example), and $\vec{d}$ is an an element in the multigraded Castelnuovo--Mumford regularity of $S/I$ (see \cite{maclagan-smith} for a detailed introduction to multigraded regularity), then the chain complex consisting of all twists less than or equal to $\vec{d} + (n_1,n_2,...,n_r)$ is a virtual resolution. They call this the \textbf{virtual resolution of the pair} $(S/I, \vec{d} \;)$. In this example, $(1,1)$ is in the multigraded regularity of $S/I$ and the virtual resolution of the pair $(S/I,(1,1))$ is shown below. 

\[S \xleftarrow{\varphi_1} \begin{matrix} S(-1,-1)^2 \\ \oplus \\ S(0,-2)^2 \\ \oplus \\ S(-2,-1) \end{matrix}
  \xleftarrow{\varphi_2} \begin{matrix} S(-1,-2)^2 \\ \oplus \\ S(-2,-2)^3 \\ \oplus \\ S(-1,-3)^2 \end{matrix}
  \xleftarrow{\varphi_3} \begin{matrix} S(-2,-3)^3 \end{matrix} \longleftarrow 0.\]

  Notice that this virtual resolution is both shorter and thinner than the minimal free resolution. This virtual resolution has nonzero first homology module $S/ \langle x_0,x_1 \rangle$, which is annihilated by $B$. Consequently, this virtual resolution is not exact. Calculating the depth of the (non-$B$-saturated) ideal of minors of each differential yields
  \[\depth(I(\varphi_1)) = 3, \hspace{1cm} \depth(I(\varphi_2)) = 2, \hspace{1cm} \depth(I(\varphi_3)) = 2,\]
  indicating again that this virtual resolution is not exact, because the depth of $I(\varphi_3)$ is less than three. Though the complex is not exact, $\depth(I(\varphi_3):B^\infty) = 3$, so it is indeed a virtual resolution.
\end{example}

\begin{remark}
In the special case of the smooth projective toric variety $\PP_{\kk}^N$, the Cox ring is $S = \kk[x_0,x_1,...,x_N]$ with the standard grading, and the irrelevant ideal is the maximal homogeneous ideal $B = \mathfrak{m} = \langle x_0,x_1,...,x_N \rangle$. If the length of the complex $F$ is not longer than $N+1$ (the bound from the Hilbert Syzygy Theorem), then Theorem~\ref{mainthm} recovers Theorem~\ref{thm:be}. This is because each $I(\varphi_i)$ is homogeneous (in fact, every minor of $\varphi_i$ is homogeneous), and for any homogeneous ideal $I \subsetneq \mathfrak{m}$,
  \[\depth(I) \le \depth(I:\mathfrak{m}^\infty) \le \depth(\sqrt{I}:\mathfrak{m}^\infty) = \depth(\sqrt{I}) = \depth(I), \]
  where the second to last equality follows from the fact that any homogeneous radical ideal other than $\mathfrak{m}$ is $\mathfrak{m}$-saturated. Therefore, condition $(b)$ in Theorem ~\ref{mainthm} becomes
  \[i \le \depth(I(\varphi_i):B^\infty) = \depth(I(\varphi_i)), \]
  which exactly matches condition $(b)$ in Theorem~\ref{thm:be}.
\end{remark}

As condition $(a)$ is the same in the two theorems, we have thus proved the following proposition. A slightly different proof is offered below.

\begin{proposition}\label{virtualexactnessprop}
  Let $S = \Cox(\PP^N_\kk) = \kk[x_0,x_1,..,x_N]$. If
\[F_\bullet :=  \left[F_0 \xleftarrow{\varphi_1} F_{1} \xleftarrow{\varphi_2} F_2 \xleftarrow{\varphi_3} \cdots \xleftarrow{\varphi_m} F_m \longleftarrow 0 \right] \]
  is a virtual resolution with $m \le N+1$, then $F_\bullet$ is exact.

  \begin{proof}
    The irrelevant ideal is $B = \langle x_0, x_1, ..., x_N \rangle$. Since $F_\bullet$ is a virtual resolution, all of its homology modules are annihilated by a power of $B$. Assume $F_\bullet$ has a nonzero homology module and let $H := H_i(F_\bullet)$ where $i$ is the largest index with $H_i(F_\bullet)$ nonzero. The Peskine--Szpiro Acyclicity Lemma \cite{ps}*{Lemma~1.8} guarantees $\depth(H) \ge 1$ and hence $B$ has a nonzerodivisor on $H$, a contradiction to the assumption that $F_\bullet$ is a virtual resolution. Therefore, in this setting, if $F_\bullet$ is a virtual resolution, it must be exact.
  \end{proof}
\end{proposition}

If the length of $F_\bullet$ is larger than $N+1$, then virtual resolutions do not need to be exact. An example is illustrated below.

\begin{example}
  Let $X = \PP^2$ and $S = \kk[x,y,z]$. Let $G_\bullet$ be the Koszul complex on $(x,y,z)$.
\[G_\bullet := \left[ S \longleftarrow S(-1)^3 \longleftarrow S(-2)^3 \longleftarrow S(-3) \longleftarrow 0 \right]. \]
  Set $F_\bullet = G_\bullet \oplus G_\bullet[-3]$. That is,
 \[F_\bullet := \left[ S \longleftarrow S(-1)^3 \longleftarrow S(-2)^3 \longleftarrow \begin{matrix} S(-3) \\ \oplus \\ S \end{matrix}
      \longleftarrow S(-1)^3 \longleftarrow S(-2)^3 \longleftarrow S(-3) \longleftarrow 0 \right]. \]
  Then the length of $F_\bullet$ is $6 > 3$ and $F_\bullet$ is a virtual resolution as the only nonzero higher homology module is $H_3(F_\bullet) = \kk$ which is annihilated by the irrelevant ideal $\langle x, y, z \rangle$.
\end{example}

\subsection*{Outline}
Section~\ref{sec:background} lays the groundwork for the rest of the paper. This includes notation, relevant definitions, and some preliminary facts about $B$-saturated homogeneous prime ideals. These primes are important, because the homogeneous localization of a module at a $B$-saturated homogeneous prime corresponds to taking the stalk of a sheaf over the toric variety $X$. Section~\ref{sec:proof} contains the proof of Theorem~\ref{mainthm}. In Section~\ref{sec:fitt}, the invariance of saturated Fitting ideals of virtual presentations is presented along with an obstruction to the number of generators of a module up to saturation. Finally, Section~\ref{sec:fittresults} contains a connection between saturated Fitting ideals and locally free sheaves. It ends with a useful result concerning unbounded virtual resolutions.  The results in Sections~\ref{sec:fitt} and~\ref{sec:fittresults} can be used to prove the reverse direction of Theorem~\ref{mainthm} in the same way that results of Fitting ideals can be used to prove Theorem~\ref{thm:be}.

\subsection*{Acknowledgements}
The author is incredibly grateful to Christine Berkesch for her guidance while this work was conducted. He would also like to thank Daniel Erman and Jorin Schug for helpful conversations and Patricia Klein, Gennady Lyubeznik, and Robert Walker for suggestions that improved the readability of this note. Finally, he also thanks an anonymous referee for a suggestion of a strengthening of Lemma~\ref{primepreservingdepth} that resulted in a simpler proof of Theorem~\ref{mainthm}. The author was partially supported by the National Science Foundation RTG grant DMS-1745638.

\section{$B$-Saturated Prime Ideals}
\label{sec:background}
In this section the structure of $B$-saturated prime ideals in the Cox ring $S$ is discussed, which will aid in later proofs. Indeed, in order to show that a complex is a virtual resolution, we will show that after sheafifying, the complex of vector bundles is acyclic by showing it is exact in each place. The latter will be true if and only if the stalk at each $B$-saturated homogeneous prime is exact.

Recall the {\bf{saturation of an ideal $I$ by $B$} } is
\[I:B^\infty = \bigcup_{n \ge 0} (I:B^n) = \{s \in S \; |\;  sB^n \subset I \text{ for some } n\}. \]
There is a correspondence between $B$-saturated ideals of $S$ and closed subschemes of $X$ \cite{CLS}*{Proposition~6.A.7}.

Throughout the paper we will often be concerned with the structure of the homogeneous $B$-saturated prime ideals of $S$, which form a proper subset of the homogeneous prime ideals of $S$. It will be important to see which prime ideals are $B$-saturated. 

\begin{proposition}[\cite{am}*{Exercise 1.12}]
  \label{bsatprimesprop}
  Suppose $P$ is a homogenous prime ideal of $S$. Then either $P$ is $B$-saturated or a prime component of $B$ is contained in $P$, in which case $P:B^\infty = S$.
\end{proposition}

\begin{proof}
  We first prove the second statement. Let $B = \bigcap_{i=1}^m Q_i$ with each $Q_i$ a homogeneous prime ideal. Then
  \[P:B^\infty = (\cdots((P:Q_1^\infty):Q_2^\infty): \cdots:Q_m^\infty).\]
  Now if $P \supset Q_i$, then $P:Q_i^\infty = P:Q_i = S$. This proves the second part of the proposition.

In order to show the first part, it suffices to show that if $P$ and $Q$ are prime ideals so that $P$ does not contain $Q$, then $P:Q^\infty = P$. Suppose $s \in P:Q^\infty$ and let $\ell \in \ZZ_{\ge 0}$ be such that $sQ^{\ell} \subset P$. Since $Q \not\subset P$, there is a $q \in Q -P$. Thus $q^{\ell} \in Q^{\ell} -P$, but $sq^{\ell} \in P$ so $s \in P$. 
\end{proof}

This first proposition says that every homogenous prime ideal of $S$ is either $B$-saturated or irrelevant. The lemma below implies we need only consider prime ideals of small enough height.

\begin{lemma}
  \label{bsatprimeslemma}
  If $I$ is a homogeneous ideal of codimension greater than $\dim X$, then $I:B^\infty = S$.
\end{lemma}

\begin{proof}
  Homogeneous ideals that saturate to all of $S$ are ideals that correspond to the empty subvariety of $X$. It is enough to show that if $I$ corresponds to a nonempty subvariety of $X$, then the height of $I$ is less than or equal to the dimension of $X$. Let $d := \dim X$, $c := \codim I$, and $N$ be the number of variables of the polynomial ring $S$. Suppose $x$ is in the subvariety of $X$ corresponding to $I$. Considering the quotient construction of $X$, there is a torus $G$ that acts on $N$-dimensional affine space $\mathbb{A}^N$. As $G$ acts freely on $\mathbb{A}^N \backslash V(B)$ \cite{CLS}*{Exercise~5.1.11}, the dimension of both $G$ and the orbit $Gx$ in $\mathbb{A}^N$ is $N-d$. Then $Gx \subset V(I) \subset \mathbb{A}^N$ implies that $c \le d$, since
  \[N-d = \dim(Gx) \le \dim V(I) = N - c. \qedhere \]
\end{proof}

Therefore, when considering homogeneous $B$-saturated primes of the Cox ring $S$ of $X$, we need only consider the homogeneous primes of codimension at most the dimension of $X$ that do not contain any prime components of the irrelevant ideal $B$.

\section{When Complexes are Virtual Resolutions}
\label{sec:proof}

In this section Theorem~\ref{mainthm} is proved. The proof relies on a lemma relating the minimum of the depths of homogeneous localization of an ideal with depth of the ideal after saturation by the irrelevant ideal $B$. There is some care that must be taken in examining degree zero component of the localization of the ideal. Fortunately, Proposition~\ref{prop:generatePic} guarantees that every element of the localization can be written as the product of a unit and a degree zero element so $S_P$ and $S_{(P)}$ may be viewed as ``the same up to units.'' In order to prove the proposition, we use the combinatorial structure of the smooth toric variety and exploit a fact about determinants of two maps in a short exact sequence.

Let $\Sigma$ be the complete fan of the smooth projective toric variety $X$, and $\sigma$ denote a cone in $\Sigma$. Then as in \cite{cox}, let
\[x^{\hat{\sigma}} = \prod_{\rho \notin \sigma} x_\rho, \]
where $x_\rho$ is the variable of the Cox ring corresponding to the ray $\rho \in \Sigma$ (see \cite{CLS}*{Section 5.2} for a more detailed exposition).
The irrelevant ideal $B$ is generated by the monomials $x^{\hat{\sigma}}$ as $\sigma$ ranges over the maximal cones of $\Sigma$.

\begin{proposition} \label{prop:generatePic}
  Let $X$ be a smooth projective toric variety. Let $\sigma$ be a maximal cone of the fan of $X$. The elements of $\Pic(X)$, $\{\deg(x_\rho) | \rho \notin \sigma\}$ generate the Picard group.
\end{proposition}

\begin{proof}
  Consider the short exact sequence
  \[0 \rightarrow L \xrightarrow{\; A \;} \Div_{T_N}(X) \xrightarrow{\; B \;} \Pic(X) \rightarrow 0\]
  of \cite{cox}*{Thm 4.3}, where $L$ is the lattice of characters of $X$ and $\Div_{T_N}(X)$ is the group of torus invariant Weil divisors of $X$. Each of $L$, $\Div_{T_N}$, and $\Pic(X)$ are free abelian groups. Here the matrix $A$ contains the rays of $\Sigma$ as rows. Since $\sigma$ is a smooth cone, the determinant of the rows corresponding to the rays of $\sigma$ is $\pm 1$. Number these rows $i_1,i_2,...,i_r$. Since the image of $B$ is the cokernel of $A$, the determinant of $B$ obtained by omitting columns $i_1,i_2,...,i_r$ also has determinant $\pm 1$. The columns we have omitted correspond to rays in $\sigma$ so omitting these columns corresponds to rays not in $\sigma$. Therefore, the degrees of the $\{x_\rho\}_{\rho \notin \sigma}$ generate $\Pic(X)$.
\end{proof}

\begin{lemma}
  \label{primepreservingdepth}
  If $I$ is a homogeneous ideal of $S$ with $I:B^\infty \ne S$, then
  \[\depth(I:B^\infty) = \min\{\depth(I_{(P)})| P \text{ is a }B\text{-saturated prime ideal of }S\}.\]
\end{lemma}

\begin{proof}
 First we claim $\depth(I_P) = \depth(I_{(P)})$ for $B$-saturated homogeneous primes $P$. As $P$ is $B$-saturated, it does not contain $x^{\hat{\sigma}}$ for some maximal cone $\sigma$. Letting $x^{\hat{\sigma}} = \prod_{\rho \notin \sigma} x_\rho$, by Proposition \ref{prop:generatePic}, $\{\deg(x_\rho)\}_{\rho \notin \sigma}$ generate $\Pic(X)$. So if $f$ is a homogeneous element of $S_P$, then there is a unit $u \in S_P$ so that $fu \in S_{(P)}$. As multiplying by a unit does not change whether or not an element is a zero divisor, this proves that $\depth(I_P) = \depth(I_{(P)})$. Then as $S$ is Cohen--Macaulay,
  \begin{align*}
    \min\{\depth(I_{P})| P \text{ is a }&B\text{-saturated prime ideal of }S\}  \\
    &= \min\{\codim(I_{P})| P \text{ is a }B\text{-saturated prime ideal of }S\} \\
    &= \min\{\codim P | I \subset P\} \\
    &= \min\{\codim P | I:B^\infty \subset P\} \\
    &= \codim(I:B^\infty) \\
    &= \depth(I:B^\infty). \qedhere
  \end{align*}
\end{proof}

\begin{proof}[Proof of Theorem \ref{mainthm}]
  Suppose first that $F_\bullet$ is a virtual resolution. Then as $\widetilde{F_\bullet}$ is exact, it is also true that $(F_\bullet)_{(P)}$ is exact for every $B$-saturated prime ideal $P$. In particular, it is true for the zero ideal, so $(F_\bullet)_{(\langle 0 \rangle)}$ is an exact sequence of vector spaces and condition $(a)$ is satisfied. Further, by Lemma \ref{primepreservingdepth}
  \[\depth(I(\varphi_i):B^\infty) = \min\{\depth(I(\varphi_i)_{(P)})| P \text{ is a }B\text{-saturated prime ideal of }S\},\]
  and  by Theorem \ref{thm:be}
  \[\min\{\depth(I(\varphi)_{(P)})| P \text{ is a }B\text{-saturated prime ideal of }S\} \ge i.\]

  Conversely, suppose conditions $(a)$ and $(b)$ are satisfied. We use a similar strategy and show that $\widetilde{F_\bullet}$ is exact by showing that $(F_\bullet)_{(P)}$ is exact for every $B$-saturated prime ideal $P$. Proposition \ref{prop:generatePic} implies that after homogeneous localization, $(F_\bullet)_{(P)}$ has the same ranks as $F_\bullet$ in every homological degree so condition $(a)$ of Theorem \ref{thm:be} is satisfied. By Lemma \ref{primepreservingdepth}, $\depth(I(\varphi_i)_{(P)}) \ge \depth(I(\varphi):B^\infty) \ge i$, so condition $(b)$ is also satisfied, proving the $(F_\bullet)_{(P)}$ is exact. Therefore $F_\bullet$ is a virtual resolution.
\end{proof}

\begin{remark}
  The same proof as above with ``homogenous'' omitted everywhere shows that in the setup of a Cohen--Macaulay ring $R$, a radical ideal $J$, and a bounded free chain complex of $R$ modules
\[F_\bullet :=  [F_0 \xleftarrow{\varphi_1} F_1 \longleftarrow \cdots \longleftarrow F_{n-1} \xleftarrow{\varphi_n} F_n \longleftarrow 0]\]
  then $F_\bullet$ satisfies $\rank(\varphi_i) + \rank(\varphi_{i+1}) = \rank(F_i)$ and $\depth(I(\varphi_i):J^\infty) \ge i$ for every $i = 1,2,...,n$ if and only if $F_\bullet$ has homology only possibly supported on $J$. That is, the complex of $\cO_{\Spec R}$-modules $\widetilde{F_\bullet}$ is exact on the open subscheme $\Spec R - V(J)$.
\end{remark}

\section{Invariance of Saturated Fitting Ideals}
\label{sec:fitt}
Between this section and Section~\ref{sec:fittresults}, we record the groundwork of tools for an alternative proof of the reverse direction of Theorem~\ref{mainthm}, using $B$-saturated Fitting ideals in the Cox ring $S$ of a smooth projective toric variety $X$. Here, we present some facts about the $B$-saturation of Fitting ideals of an $S$-module. The Fitting ideals of a finitely-generated module $M$ over a Noetherian ring can be calculated by a free presentation
\[F \xrightarrow{\; \varphi \;} G \longrightarrow M \longrightarrow 0.\]
Let $r$ denote the rank of $G$. The {\textbf{$i$th Fitting ideal of $M$}}, $\Fitt_i(M)$, is defined to be $I_{r-i}(\varphi)$. This ideal is independent of the free presentation \cite{Fitting}.

We now adapt this idea of Fitting ideals to virtual presentations of a $\Pic(X)$-graded $S$-module $M$, beginning with the invariance of the $B$-saturated Fitting ideals. We then produce further facts about saturated Fitting ideals that mirror the classical theory of Fitting ideals.

Let a chain complex
\[F \xrightarrow{\; \varphi \;} G \longrightarrow M \longrightarrow 0\]
be called a \textbf{virtual presentation} if both $F$ and $G$ are free $S$-modules and $\widetilde{\coker(\varphi)} \cong \widetilde{M}$.

\begin{lemma}[Saturated Fitting's Lemma]
  \label{satinv}
  Suppose $X$ is a smooth projective toric variety and $S$ is the Cox ring of $X$ with irrelevant ideal $B$. Let
  \[F \xrightarrow{\phantom{X}\varphi} G \longrightarrow M \longrightarrow 0 \hspace{1cm} \text{ and } \hspace{1cm} F' \xrightarrow{\phantom{X}\varphi'} G' \longrightarrow M' \longrightarrow 0\]
  be finite virtual presentations of $M$ and $M'$, respectively, with $G$ and $G'$ of ranks $r$ and $r'$. If $\widetilde{M} \cong \widetilde{M'}$, then $I_{r-i}(\varphi):B^\infty = I_{r'-i}(\varphi'):B^\infty$ for every $i \in \ZZ_{\ge 0}$.
\end{lemma}

\begin{proof}
  We may harmlessly assume that $M$ is already $B$-saturated and that
  \[F \xrightarrow{\; \varphi \;} G \longrightarrow M \longrightarrow 0\]
  is the truncation of the minimal free resolution of $M$. By replacing $M'$ with $\coker \varphi'$ if necessary, we may also assume that
  \[F' \xrightarrow{\; \varphi' \;} G' \longrightarrow M' \longrightarrow 0\]
  is a free presentation of $M'$.

  Now $I_{r-i}(\varphi):B^\infty = I_{r'-i}(\varphi'):B^\infty$ if and only if the sheaves $\widetilde{I_{r-i}(\varphi)}$ and $\widetilde{I_{r'-i}(\varphi')}$ are equal. We will show the equality of the sheaves by showing that they agree on an open cover of $X$ by some distinguished open affines $D(f)$. The cocycle condition is then satisfied because taking Fitting ideals commutes with base change. Given $f \in S$, $D(f)$ can be written (by abusing notation) as $\{p \in X | f(p) \ne 0\}$.

  The open cover will be by the generators $\{x^{\hat{\sigma}}\}_{\sigma \in \Sigma}$ of the irrelevant ideal. To see that this forms an open cover of $X$, begin by fixing a $B$-saturated prime $P$. We shall show that $P \in D(x^{\hat{\sigma}})$ for some $\sigma$. Notice $P \in D(x^{\hat{\sigma}})$ exactly when $x^{\hat{\sigma}} \notin P$. If $x^{\hat{\sigma}} \in P$ for every $\sigma \in \Sigma$, then $P$ contains the irrelevant ideal $B$ and is therefore not $B$-saturated.

By Proposition \ref{prop:generatePic}, homogeneous localization by $x^{\hat{\sigma}}$ preserves free modules and so 
  \[\widetilde{I_{r-i}(\varphi)}(D(x^{\hat{\sigma}})) \cong I_{r-i}(\varphi)_{(x^{\hat{\sigma}})} = I_{r-i}(\varphi_{(x^{\hat{\sigma}})}),\]
  where the last equality holds because taking Fitting ideals commutes with base change. Because localization is exact, and Proposition \ref{prop:generatePic} implies that every element of $F_{x^{\hat{\sigma}}}$ is equal to a product of a unit and an element of $F_{(x^{\hat{\sigma}})},$
  \[F_{(x^{\hat{\sigma}})} \xrightarrow{\varphi_{(x^{\hat{\sigma}})}} G_{(x^{\hat{\sigma}})} \longrightarrow M_{(x^{\hat{\sigma}})} \longrightarrow 0\]
  is a free presentation of $M_{(x^{\hat{\sigma}})}$.
  
  Furthermore, since $\widetilde{M} = \widetilde{M'}$, we have
  \[M_{(x^{\hat{\sigma}})} \cong \widetilde{M}(D(x^{\hat{\sigma}})) = \widetilde{M'}(D(x^{\hat{\sigma}})) \cong M'_{(x^{\hat{\sigma}})}.\]
  Fitting's Lemma \cite{Fitting} implies $I_{r-i}(\varphi_{(x^{\hat{\sigma}})}) = I_{r'-i}(\varphi'_{(x^{\hat{\sigma}})})$.

Thus $\widetilde{I_{r-i}(\varphi)}$ and $\widetilde{I_{r'-i}(\varphi')}$ agree on an open cover of $X$, and glue in the same way so the sheaves are equal.
\end{proof}

We will call $I_{r-i}(\varphi):B^\infty$ the \textbf{$i$th saturated Fitting ideal of $M$}. Lemma~\ref{satinv} allows us to prove that the $j$th saturated Fitting ideal not equaling $S$ is an obstruction to generating a module up to saturation by $j$ elements. It is analogous to the fact that the $j$th Fitting ideal not equaling $S$ is an obstruction to generating a module by $j$ elements in the classical theory of Fitting ideals of Noetherian rings (see \cite{eisenbud-ca}*{Proposition~20.6}).

Define $\mathbb{V}(I)$ to be the set of homogeneous $B$-saturated prime ideals $P$ of $S$ such that $P \supset I$. In particular,
\[\mathbb{V}(I) := \{P \in V(I) \; | \; P \text{ is } B\text{-saturated}\}.\]

\begin{proposition}
  \label{vfitt}
  The set $\mathbb{V}(\Fitt_j(M):B^\infty)$ consists of exactly the homogeneous $B$-saturated primes $P$ such that there does not exist an $S$-module $N$ where $N_P$ can be generated by $j$ elements and $\widetilde{M} \cong \widetilde{N}$.
\end{proposition}

\begin{proof}
  Suppose $P \in \mathbb{V}(\Fitt_j(M):B^\infty)$ and $N$ is an module such that the sheaves $\widetilde{M}$ and $\widetilde{N}$ are isomorphic. Then by Lemma~\ref{satinv}, for every $j \in \ZZ_{\ge 0}$, $\Fitt_j(M):B^\infty = \Fitt_j(N):B^\infty$. So
  \[\Fitt_j(N) \subset \Fitt_j(N):B^\infty \subset P,\]
and therefore $N_P$ cannot be generated by $j$ elements \cite{eisenbud-ca}*{Proposition~20.6}.

  On the other hand, suppose $P$ is a homogeneous $B$-saturated prime not belonging to $\mathbb{V}(\Fitt_j(M):B^\infty)$. In this case, $\Fitt_j(M)$ cannot be contained in $P$. For if $\Fitt_j(M) \subset P$, then $\Fitt_j(M):B^\infty \subset P:B^\infty = P$. Thus $M_P$ can be generated by $j$ elements.
\end{proof}

As mentioned above, the previous proposition can be rephrased as saying that $\Fitt_j(M)$ is an obstruction to generating the sheaf $\widetilde{M}$ by $j$ elements. One may hope that the converse is true as well, that is, if $\Fitt_j(M):B^\infty = S$, then there is a module $N$ generated by $j$ elements such that $\widetilde{M} \cong \widetilde{N}$. Unfortunately, this fails to be true as shown by the following example.

\begin{example}
  Let $X = \PP^1 \times \PP^1$ so that $S = \kk[x_0,x_1,y_0,y_1]$ with $\deg(x_i) = (1,0)$ and $\deg(y_i) = (0,1)$. Consider the $B$-saturated ideal $I = \langle y_0y_1,x_0y_0,x_0x_1 \rangle$ of three points lying in the following way on a ruling of $\PP^1 \times \PP^1$.
  
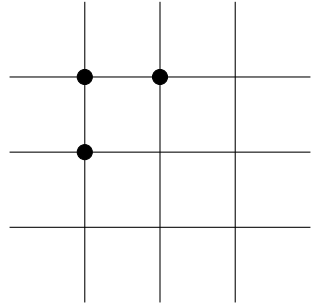
\begin{center}
  \begin{tikzpicture}
    \draw (-1,0) -- (3,0);
    \draw (-1,1) -- (3,1);
    \draw (-1,2) -- (3,2);
    \draw (0,-1) -- (0,3);
    \draw (1,-1) -- (1,3);
    \draw (2,-1) -- (2,3);
    \draw[fill] (0,2) circle [radius=0.1];
    \draw[fill] (1,2) circle [radius=0.1];
    \draw[fill] (0,1) circle [radius=0.1];
  \end{tikzpicture}
\end{center}
Then $S/I$ has minimal resolution
\[0 \longrightarrow \begin{matrix} S(-2,-1) \\ \oplus \\ S(-1,-2) \end{matrix}
\xrightarrow{\begin{bmatrix} -y_0 & 0 \\ x_1 & -y_1 \\ 0 & x_0 \end{bmatrix}}
\begin{matrix} S(-2,0) \\ \oplus \\ S(-1,-1) \\ \oplus \\ S(0,-2) \end{matrix}
\longrightarrow S.\]
  
The ideal $\Fitt_2(S/I)$ is generated by the entries of the matrix above, so
\[\Fitt_2(S/I):B^\infty = S.\]

However, there cannot be two homogenous polynomials of $S$ whose intersection vanishes exactly at the above three points. For if $f$ and $g$ are homogeneous forms of degree $(a,b)$ and $(c,d)$ respectively, the multigraded version of B\'{e}zout's Theorem (see example 4.9 in \cite{shaf}) says that the intersection multiplicity of $f$ and $g$ is $ad + bc$. For concreteness, the leftmost vertical line is $V(x_0)$, the middle vertical line is $V(x_1)$, the topmost horizontal line is $V(y_0)$ and the middle horizontal line is $V(y_1)$. Then there are two cases for which $ad+bc = 3$.

\textit{Case 1:} One of the terms ($ad$ or $bc$) is 3 and one is 0. Without loss of generality, assume $a = 3$, $d = 1$. If $c = 0$, then $\deg(g) = (0,1)$, which cannot vanish at all three points. Hence, $b = 0$ and $\deg(f) = (3,0)$. Plugging in $x_0 = 0$, $x_1 = 1$ gives a degree $0$ form that must vanish at $[0:1]$ and $[1:0]$ which means $f$ must vanish on all of $V(x_0)$. Similarly, plugging in $x_0 = 0$, $x_1 = 1$ to $g$ yields a degree 1 form that also vanishes at these two points, implying $g$ vanishes on all of $V(x_0)$. Therefore $V(f) \cap V(g) \ne X$.

\textit{Case 2:} One of the terms is 2, and the other is equal to 1. It suffices to assume $a = 2$ and $b = c = d = 1$. Then both $f$ and $g$ restricted to $V(x_0)$ are degree one polynomials in $y_0$ and $y_1$ that vanish at two points so again $V(x_0) \subset V(f) \cap V(g)$. Therefore, no ideal generated by two elements will saturate to the ideal $I$.
\end{example}

\section{Saturated Fitting Ideals and Locally Free Sheaves}
\label{sec:fittresults}

In the classical theory of Fitting ideals, a module is projective if and only if its first nonzero Fitting ideal is the whole ring. Sheafifying a projective module yields a locally free sheaf. The situation is similar for modules over the Cox ring and saturated Fitting ideals.

\begin{proposition}
  \label{localfree}
  If $M$ is a $\Pic(X)$-graded $S$-module, then $\Fitt_r(M):B^\infty = S$ and $\Fitt_{r-1}(M):B^\infty = 0$ if and only if the sheaf $\widetilde{M}$ is locally free of constant rank $r$.
\end{proposition}

\begin{proof}
  First, if $\widetilde{M}$ is locally free of rank $r$, then for any $B$-saturated prime $P$, $M_{(P)} \cong \widetilde{M}_P$ is a free $S_{(P)}$-module of rank $r$. Since $P$ is $B$-saturated, there is some $x^{\hat{\sigma}}$ so that $x^{\hat{\sigma}} \notin P$, which implies $M_P = S_PM_{(P)}$. Therefore, $M_P$ has the free presentation
  \[S_P^r \longleftarrow 0,\]
  so $\Fitt_r(M_P) = S_P$ and $\Fitt_{r-1}(M_P) = 0$. Since taking Fitting ideals commutes with localization, $\Fitt_{r-1}(M) = 0$ and for every $B$-saturated prime $P$, $\Fitt_{r}(M) \not\subset P$. By definition,
  \[\codim \Fitt_r(M) = \sup_{P \in \, \Spec S} \{\codim P: \Fitt_r(M) \subset P \}\]
  is strictly greater than the codimension of any $B$-saturated prime (see Section~\ref{sec:background}). The maximal codimension of any $B$-saturated prime of $S$ is $\dim X$ by Lemma~\ref{bsatprimeslemma}. So again by Lemma~\ref{bsatprimeslemma}, $\Fitt_r(M):B^\infty = S$ as desired.

  Now suppose $\Fitt_r(M):B^\infty = S$ and $\Fitt_{r-1}(M):B^\infty = 0$, and let $P$ be a $B$-saturated prime ideal of $S$. By Proposition~\ref{vfitt}, there is an $S$-module $N$ with $\widetilde{N} \cong \widetilde{M}$ such that $N_P$ can be generated by $r$ elements over $S_P$. Let $\varphi_P\colon  G_P \rightarrow S_P^r$ be a free presentation of $N_P$. Then
  \[I_1(\varphi_P) = \Fitt_{r-1}(N_P) \subset \Fitt_{r-1}(N_P):B^\infty = \Fitt_{r-1}(M_P):B^\infty = 0.\]
  Thus $\varphi_P = 0$, which implies $N_P \cong S_P^r$ and $N_{(P)} \cong S_{(P)}^r$. Since each the stalk at each point of the sheaf $\widetilde{N}$ is free of rank $r$, this means $\widetilde{N}$ is locally free of rank $r$, as desired.
\end{proof}

Notice that in everything up to this point, we have required that the complex be bounded. This may seem unsatisfying as Definition~\ref{def:vres} does not require the length of a virtual resolution to be finite.  The following result does not require this hypothesis, again mirroring the classical theory of Fitting ideals and exactness.

\begin{proposition}
  A complex of free $S$-modules
  \[F \xrightarrow{\varphi} G \xrightarrow{\psi}H\]
  with $I(\varphi):B^\infty = S = I(\psi):B^\infty$ has irrelevant homology (i.e. the homology is supported only on $B$) if and only if $\rank(\varphi) + \rank(\psi) = \rank(G)$.
\end{proposition}

\begin{proof}
  The sequence
  \[F \xrightarrow{\; \varphi\; } G \rightarrow \coker \varphi \rightarrow 0\]
   is right exact and a free presentation of $\coker \varphi$. Letting $\rank(G) = n$ and $\rank(\varphi) = r$,
  \[\Fitt_{n-r}(\coker \varphi):B^\infty = I_r(\varphi):B^\infty = S, \text{ and } \]
  \[  \Fitt_{n-r-1}(\coker \varphi):B^\infty = I_{r+1}(\varphi):B^\infty = 0.\]
By Proposition~\ref{localfree}, $\widetilde{\coker \varphi}$ is a locally free sheaf of constant rank $n-r$. As localization is exact, homogeneous localization at a $B$-saturated prime $P$ gives
  \[F_{(P)} \rightarrow G_{(P)} \rightarrow \coker(\varphi_{(P)}) \rightarrow 0.\]
  The rank of $\coker \varphi_{(P)}$ is equal to $\rank(G_{(P)}) - \rank(\varphi_{(P)})$. Notice that $\rank(\varphi_{(P)}) = \rank(\varphi)$ since $I(\varphi)$ contains a nonzerodivisor. Now since
  \[F_{(P)} \xrightarrow{\varphi_{(P)}} G_{(P)} \xrightarrow{\psi_{(P)}} H_{(P)}\]
  is a complex, $\psi_{(P)}$ factors through $\coker \varphi_{(P)}$:
  \begin{center}
  \begin{tikzcd}
    G_{(P)} \arrow[rr, "\psi_{(P)}"] \arrow[dr]& & H_{(P)} \\
    & \coker \varphi_{(P)} \arrow[ur, "\psi_{(P)}'"] &
  \end{tikzcd}
\end{center}
As $\coker \varphi_{(P)}$ is free of rank $n-r$, $\rank(\psi_{(P)}) = \rank(\psi_{(P)}')$ and $I_j(\psi_{(P)}) = I_j(\psi_{(P)}')$ for every $j \in \ZZ_{\ge 0}$.

Now $F_{(P)} \rightarrow G_{(P)} \rightarrow H_{(P)}$ is exact if and only $G_{(P)} \xrightarrow{0} \coker \varphi_{(P)} \xrightarrow{\psi_{(P)}'} H_{(P)}$ is exact in which case $\rank(\psi_{(P)}') = \rank(\coker \varphi_{(P)})$. Assembling the equalities shows
\begin{align*}
  \rank(\varphi) + \rank(\psi) &= \rank(\varphi_{(P)}) + \rank(\psi_{(P)}')\\
                               &= \rank(\varphi_{(P)}) + \rank(G_{(P)}) - \rank(\varphi_{(P)})\\
                               &= \rank(G). \qedhere
                                 \end{align*}
\end{proof}

This proposition can be used in an alternate proof of the reverse direction of Theorem~\ref{mainthm}.

\begin{proof}
  We prove the conditions $(a)$ and $(b)$ imply the complex
\[F_\bullet :=  [0 \rightarrow F_n \xrightarrow{\varphi_n} F_{n-1} \longrightarrow \cdots \longrightarrow F_1 \xrightarrow{\varphi_1} F_0] \]
is a virtual resolution.  

It is enough to show for each $B$-saturated homogeneous prime ideal $P$ of $S$, the complex $\widetilde{(F_\bullet)}_P$ is exact. This is because $\widetilde{(F_\bullet)}_P$ is a graded complex so if it is exact it will remain exact after taking the degree zero strand $\widetilde{(F_\bullet)}_{(P)}$. Since $S$ is a polynomial ring, it is an integral domain, so each $I(\varphi_i)$ contains a nonzerodivisor. Therefore, as localization commutes with taking Fitting ideals \cite{eisenbud-ca}*{Corollary~20.5}, the hypotheses are not weakened. We will show that $(\widetilde{F_\bullet})_P$ is exact by induction on the codimension of $P$. The unique minimal prime ideal of $S$ is $\langle 0 \rangle$, and localizing at this ideal gives a complex of vector spaces. The complex $(F_\bullet)_{\langle 0 \rangle}$ therefore becomes exact by assumption $(a)$.

  Now suppose $P$ is a $B$-saturated prime of codimension $c > 0$. By assumption $(b)$, $\depth(I(\varphi_{c+1}):B^\infty) \ge c+1$, so $I(\varphi_{c+1}):B^\infty$ is not contained in $P$. Therefore $I(\varphi_{c+1})$ is also not contained in $P$. Indeed, if $I(\varphi_{c+1}) \subset P$, then $I(\varphi_{c+1}):B^\infty \subset P:B^\infty = P$. Hence $I(\varphi_{c+1})_P = S_P$, so $(F_\bullet)_P$ can be broken into two complexes:
    \[(G_\bullet)_P \colon \quad  0 \longrightarrow (F_n)_P \longrightarrow \cdots (F_{c+1})_P \xrightarrow{(\varphi_{c+1})_P} (F_{c})_P \longrightarrow (\coker \varphi_{c+1})_P \longrightarrow 0,\]
      \[(H_\bullet)_P \colon \quad 0 \longrightarrow (\coker \varphi_{c+1})_P \xrightarrow{(\varphi_{c}')_P} (F_{c-1})_P \xrightarrow{(\varphi_{c-1})_P} (F_{c-2})_P \longrightarrow \cdots \longrightarrow (F_0)_P,\]
      where $\varphi_{c}'$ is induced from $\varphi_{c}$. Now since $I((\varphi_i)_P) = S_P$ for every $i \ge c+1$, the complex $(G_\bullet)_P$ is exact. All that is left to show is that $(H_\bullet)_P$ is exact. Since $I((\varphi_{c+1})_P) = S_P$, the $S_P$-module $(\coker \varphi_{c+1})_P$ is projective and hence free. Also notice that $I((\varphi_{c}')_P) = I((\varphi_{c})_P)$. 

      By induction, $(H_\bullet)_P$ becomes exact when localizing at any homogeneous $B$-saturated prime ideal $Q$ properly contained in $P$. So by Proposition~\ref{bsatprimesprop}, it becomes exact when localizing at any prime ideal contained in $P$ (if a smaller prime contains a prime component of $B$, then $P$ does as well, which means that $P$ would not be $B$-saturated). Therefore the homology modules are only supported on the maximal ideal $P_P$ and thus have depth 0. The depth of each free $S_P$-module is equal to the codimension of $P$ (since $S$ is Cohen--Macaulay), which is strictly positive. Therefore, applying the Peskine--Szpiro Acyclicity Lemma \cite{ps}*{Lemma~1.8} completes the proof.
\end{proof}


 \raggedright


\begin{bibdiv}
\begin{biblist}
\bib{am}{book}{
   author={Atiyah, M. F.},
   author={Macdonald, I. G.},
   title={Introduction to commutative algebra},
   publisher={Addison-Wesley Publishing Co., Reading, Mass.-London-Don
   Mills, Ont.},
   date={1969},
   pages={ix+128},
   review={\MR{0242802}},
}
  
\bib{bes}{article}{
author = {Berkesch, Christine},
author = {Erman, Daniel},
author = {Smith, Gregory G.},
title = {Virtual Resolutions for a Product of Projective Spaces},
journal = {Alg. Geom.},
fjournal = {Algebraic Geometry},
volume = {7},
year = {2020},
number = {4},
pages = {460--481}
}

\bib{be}{article}{
  author = {Buchsbaum, David A.},
  author = {Eisenbud, David},
  title = {\href{https://doi.org/10.1016/0021-8693(73)90044-6}%
    {What Makes a Complex Exact?}},
  journal = {J. Algebra},
  volume = {25},
  date = {1973},
  pages = {259-268}
}

	
\bib{cox}{article}{
  author={Cox, David A.},
  title={\href{https://arxiv.org/abs/alg-geom/9210008}%
    {The homogeneous coordinate ring of a toric variety}},
  journal={J. Algebraic Geom.},
  volume={4},
  date={1995},
  number={1},
  pages={17--50},
}

\bib{CLS}{book}{
   author={Cox, David A.},
   author={Little, John B.},
   author={Schenck, Henry K.},
   title={\href{http://dx.doi.org/10.1090/gsm/124}%
     {Toric varieties}},
   series={Graduate Studies in Mathematics~124},
   publisher={Amer. Math. Soc., Providence, RI},
   date={2011},
   pages={xxiv+841},
}


\bib{eisenbud-ca}{book}{
  author={Eisenbud, David},
  title={\href{http://dx.doi.org/10.1007/978-1-4612-5350-1}%
    {Commutative algebra with a view toward algebraic geometry}},
  series={Graduate Texts in Mathematics~150},
  publisher={Springer-Verlag, New York},
  date={1995},
  pages={xvi+785},
}

\bib{Fitting}{article}{
author = {Fitting, H.},
journal = {Jahresbericht der Deutschen Mathematiker-Vereinigung},
keywords = {Abstract theory of rings, fields, etc.},
pages = {195-228},
title = {\href{http://eudml.org/doc/146122}%
  {Die Determinantenideale eines Moduls.}},
volume = {46},
year = {1936},
}


 \bib{ps}{article}{
   author={Peskine, C.},
   author={Szpiro, L.},
   title={Dimension projective finie et cohomologie locale}
   journal={Inst. Hautes \'{E}tudes Sci. Publ. Math}
   volume={42}
   date={1973}
   pages={47-119},
 }


\bib{maclagan-smith}{article}{
   author={Maclagan, Diane},
   author={Smith, Gregory G.},
   title={\href{http://dx.doi.org/10.1515/crll.2004.040}%
     {Multigraded Castelnuovo-Mumford regularity}},
   journal={J. Reine Angew. Math.},
   volume={571},
   date={2004},
   pages={179--212},
}

\bib{M2}{misc}{
  label={M2},
  author={Grayson, Daniel~R.},
  author={Stillman, Michael~E.},
  title={Macaulay2, a software system for research
    in algebraic geometry},
  publisher = {available at \url{http://www.math.uiuc.edu/Macaulay2/}},
}

\bib{shaf}{book}{
  author={Shafarevich, Igor~R.},
  title={\href{http://dx.doi.org/10.1007/978-3-642-37956-7}%
    {Basic Algebraic Geometry I}},
  subtitle={Varieties in Projective Space},
  publisher={Springer-Verlag, New York-Heidelberg},
  date={1977},
  pages={247},
  edition={3rd},
}


 \end{biblist}
\end{bibdiv}
\end{document}